\documentclass[a4paper,11pt]{article}

\usepackage{mathrsfs,epsfig,psfrag}
\usepackage{amsmath,amsthm}
\usepackage{amssymb,enumerate}

\reversemarginpar
\makeatletter
\renewcommand\section{\@startsection {section}{1}{\z@}%
{-3.5ex \@plus -1ex \@minus -.2ex}%
{2.3ex \@plus.2ex}%
{\raggedright\normalfont\large\bfseries}}
\makeatother

\makeatletter
\renewcommand\subsection{\@startsection {subsection}{1}{\z@}%
{-3.5ex \@plus -1ex \@minus -.2ex}%
{2.3ex \@plus.2ex}%
{\raggedright\normalfont\bfseries}}
\makeatother

\newtheorem*{thmA}{Theorem~A}
\newtheorem*{corA}{Corollary~A}
\newtheorem*{thmB}{Theorem~B}

\newtheorem{lemma}{Lemma}

\newtheorem*{pbS}{Problem (S)}
\newtheorem*{pbL}{Problem (L)}
\newtheorem*{pbC}{Problem (C)}

\theoremstyle{definition}
\newtheorem{Def}{Definition}

\theoremstyle{remark}
\newtheorem{rem}{Remark}

\def\dst{\displaystyle}

\def\ens{\enspace}

\def\al{\alpha}
\def\be{\beta}
\def\ga{\gamma}
\def\Ga{{\Gamma}}
\def\de{\delta}

\def\De{\Delta}
\def\eps{{\varepsilon}}
\def\ka{\kappa}
\def\la{\lambda}
\def\La{\Lambda}

\def\Om{\Omega}

\def\th{{\theta}}

\newcommand{\ph}{\varphi}

\def\ze{{\zeta}}

\newcommand{\demi}{\frac{1}{2}}

\newcommand{\ao}{\{\,}          
\newcommand{\af}{\,\}}          

\def\ov{\overline}

\newcommand{\dist}{\operatorname{dist}}

\newcommand{\length}{\operatorname{length}}
\newcommand{\ID}{\operatorname{Id}}

\def\pa{\partial}

\def\ii{^{-1}}
\def\ti{\tilde}

\def\IM{\mathop{\Im m}\nolimits}
\def\RE{\mathop{\Re e}\nolimits}

\def\ie{{i.e.}}

\def\eg{{e.g.}}

\def\wrt{{with respect to}}

\def\C{\mathbb{C}}
\def\D{\mathbb{D}}
\def\E{\mathbb{E}}
\def\N{\mathbb{N}}
\newcommand{\PP}{\widehat{\mathbb{C}}}
\def\Q{\mathbb{Q}}
\def\R{\mathbb{R}}
\renewcommand{\SS}{\mathbb{S}}

\def\Z{\mathbb{Z}}

\def\cA{\mathcal{A}}

\def\cE{\mathcal{E}}
\def\cF{\mathcal{F}}

\def\cI{\mathcal{I}}
\def\cJ{\mathcal{J}}

\def\cL{\mathcal{L}}

\def\cN{\mathcal{N}}
\def\cR{\mathcal{R}}

\newcommand{\I}{{\mathrm i}}
\newcommand{\dd}{{\mathrm d}}

\def\ee{{\mathrm e}}
\def\e#1{\ee^{#1}}

\DeclareMathOperator{\DC}{DC}
\DeclareMathOperator{\Exp}{Exp}
\newcommand{\DCgt}{\DC_{\ga,\tau}}

\newcommand{\ld}{\langle}
\newcommand{\rd}{\rangle}

\def\pp#1{^{(#1)}}

\def\norm#1{\Vert#1\Vert}

\newcommand{\IN}{^{\textnormal{(i)}}}
\newcommand{\EX}{^{\textnormal{(e)}}}

\newcommand{\gB}{\mathscr B}
\newcommand{\gC}{\mathscr C}
\newcommand{\gF}{\mathscr F}
\newcommand{\gH}{\mathscr H}
\newcommand{\gM}{\mathscr M}
\newcommand{\gO}{\mathscr O}

\newcommand{\gV}{\mathscr V}
\newcommand{\Chol}{\gC^1_{\textrm{hol}}}
\newcommand{\Cinfhol}{\gC^\infty_{\textrm{hol}}}

\newcommand{\ovci}[1]{\overset{\circ}{#1}}
\newcommand{\ovint}[2]{\ovci{\raisebox{0ex}[#2ex]{$#1$}}}
\newcommand{\str}[1]{\!\raisebox{0ex}[#1ex]{\,}}

\newcommand{\FF}{\cF_{q,\eps}}

\newcommand{\tn}[1]{\textnormal{(#1)}}

\begin{document}

\thispagestyle{empty}

\begin{center}

  {\LARGE\bf A quasianalyticity property for monogenic\\[1ex]
      solutions of small divisor problems } 

\bigskip

\bigskip

{\bf Stefano Marmi, David Sauzin} {\small (16 September 2010)}

\end{center}

\bigskip

\bigskip


\begin{abstract}
  We discuss the quasianalytic properties of various spaces of functions
  suitable for one-dimensional small divisor problems.
  These spaces are formed of functions $\gC^1$-holomorphic on certain compact
  sets~$K_j$ of the Riemann sphere (in the Whitney sense), as is the solution of
  a linear or non-linear small divisor problem when viewed as a function of the
  multiplier (the intersection of~$K_j$ with the unit circle is defined by a
  Diophantine-type condition, so as to avoid the divergence caused by roots of
  unity).
  It turns out that a kind of generalized analytic continuation through the unit
  circle is possible under suitable conditions on the $K_j$'s.
\end{abstract}

\bigskip


\setcounter{section}{-1}

\section{Introduction}

Following V.~Arnold and M.~Herman, and in the same line of research as in
\cite{MS1}, we consider ``monogenic functions'' in the sense of \'Emile Borel
with a view to small divisor problems.
In these problems of dynamical origin, there is a complex parameter~$q$, called
multiplier, which must be kept off the roots of unity in order to solve a
functional equation; typically, $q$~is the eigenvalue at a fixed point of a
one-dimensional complex map that one wants to linearize and one studies the
equation (corresponding to the so-called Siegel problem)
\begin{equation}	\label{eqSieg}
h(q,qz) = q G\bigl(h(q,z)\bigr)
\end{equation}
(where $G(z)\in z\C\{z\}$ is given, with $G'(0)=1$, 
and $h(q,\,.\,)$ is sought in a Banach space of functions holomorphic in the variable~$z$),
or the linearized equation $h(q,qz) - q h(q,z) = q g(z)$ (with $g(z)\in
z^2\C\{z\}$ given),
or the more complicated non-linear equation corresponding to the conjugacy
between a circle map and a rigid rotation with rotation number
$\frac{1}{2\pi\I}\log q$ (see equation~(\ref{eqpbC})).

We are interested in the dependence of the solution on the multiplier~$q$.
Roots of unity act as resonances, because the coefficients of the
solution of the problem are inductively defined by expressions which involve
division by $q^k-1$, $k\ge1$. 
On the other hand the case where $|q|=1$ is particularly interesting from the
dynamical point of view.
One is thus led to define compact sets~$K_j$ of
the Riemann sphere~$\PP$ by removing smaller and smaller neighbourhoods of the
roots of unity.
It is shown in \cite{He} and \cite{CM} for the above-mentioned non-linear
problems and in \cite{MS1} for the linear one, that the
solution is Whitney smooth on the $K_j$'s, which gives rise to an example of
``monogenic'' function (the definition is recalled in Section~\ref{secMonFcn}).
In all the cases we consider, the union of the~$K_j$'s on which our monogenic
functions are defined contains $\PP\setminus\SS$, where~$\SS$ denotes the unit
circle, and also a subset of~$\SS$ defined by an arithmetical condition (Bruno
or Diophantine condition); in restriction to $\PP\setminus\SS$, the functions
are analytic in the usual sense.

From the point of view of classical analytic continuation, the unit circle~$\SS$
appears as a natural boundary, because of the density of the roots of unity, but
the question arises whether ``monogenic continuation'' through~$\SS$ is possible.
A related important issue, as emphasized by M.~Herman, is that of
quasianalyticity: Is a monogenic function determined by its Taylor series at a
point?  And indeed the Taylor series is well defined at points of~$\SS$
corresponding to a Diophantine-type condition, but this series is divergent.
At the end of \cite{He}, Herman writes: ``we believe that \'E.~Borel (\ldots)
  wanted his monogenic functions to have quasianalytic properties
  (\ie\ monogenic continuation),''
but ``the (solution of the) linearized equation does not seem to belong to any
quasianalytic class''.
This is confirmed by our work \cite{MS1} (see also Remark~\ref{remNonQA} in Section~\ref{secStatRes}).

The question of quasianalyticity can also be raised at each point of
$\PP\setminus\SS$ (where convergent Taylor series are available) and, though
easier to answer, it is still non-trivial, because the domain of analyticity is
not connected.

Instead of the traditional notion of quasianalyticity, one
may consider a weaker property: 
we shall speak of ``$\gH^1$-quasianalyticity'' whenever the functions are
determined by their restriction to any subset of positive linear Hausdorff
measure (see Section~\ref{secVarQA}).
%
%
The subject of the Part~A of this article is to prove such a property for spaces
of monogenic functions defined on compact sets of~$\PP$ of a certain kind
(Section~\ref{secStatRes}); we shall see in Part~B that these spaces are large
enough to contain the monogenic functions which appear in small divisor problems,
so that we obtain a form of monogenic continuation across the unit circle \wrt\
the multiplier.
More specifically, the small divisor problems considered in Part~B are: 
\begin{enumerate}[--]
\item The Siegel problem~(\ref{eqSieg}), the solution of which is shown to be monogenic in
\cite{CM}, relatively to compact sets~$K_j$ described in
Section~\ref{secSDdom}; their union intersects~$\SS$ along a set corresponding
to the Bruno condition (optimal for this problem).
\item The linearized problem, the solution of which is shown to be monogenic in
\cite{MS1}, relatively to compact sets described in Section~\ref{secSDdom}.
\item The complexified local conjugacy problem of circle maps described under the name
Problem~\tn{C} in Section~\ref{secIntroSDpb}.
\end{enumerate}
In the last case, one is given a family of maps of the form 
$\th \mapsto G_{\al,\eps}(\th) = \th + \al + \eps g(\th)$
with a holomorphic $1$-periodic function~$g$ of zero mean-value.
The relevant multiplier turns out to be $q=\ee^{2\pi\I\al}$, while $\eps$ is here a
small complex parameter.
The equation to be solved is
\begin{equation}	\label{eqpbC}
u(\th+\al) - u(\th) + \be = \eps g\big(\th+u(\th)\big),
\end{equation}
where one looks for $\be\in\C$ and a holomorphic $1$-periodic function~$u$ of
zero mean-value.
As explained in Section~\ref{secIntroSDpb}, this amounts to
conjugating~$G_{\al-\be,\eps}$ to the rigid rotation~$G_{\al,0}$.
We speak of complexified problem because $\al$, $\th$ and $\eps g$ are not
assumed to be real.

The question of the monogenic regularity of~$(\be, u)$ \wrt~$\al$ (or,
equivalently, \wrt~$q$) 
was raised by Arnold in \cite{Ar} without an answer, due to limitations of the method employed
there. Later, in \cite{He}, Herman proved the monogenicity of the solution relatively to
compact sets defined by means of a Diophantine condition\footnote{
We followed quite closely \cite{He} and sticked to this Diophantine condition,
although one could have tried to adapt the results of
\cite{Ri} which deals with the Bruno condition (optimal for this problem) by
means of Yoccoz's renormalization method.}
and contained in a narrow strip $\{ |\IM\al|<\rho \}$ in the complex domain.

We shall be able to extend Herman's regularity result up to domains~$K_j$ of the
kind which is required to apply the quasianalyticity result of Part~A.
Indeed, when complexifying the problem of real circle maps, it may seem natural
to focus on a strip for~$\al$, which corresponds to a neighbourhood of~$\SS$ in~$\PP$ for
the multiplier~$q$, but it is important for our quasianalyticity results that
the domains for~$q$ extend up to~$0$ and~$\infty$.

As a consequence, we obtain for instance that the solution of any of the
three above-mentioned small divisor
problems with any given multiplier is (at least theoretically) determined by the
solution of the same problem for a small set of values of the multiplier
(provided this set has positive linear measure) or by the Taylor series of the
solution \wrt\ the multiplier at any other point of $\PP\setminus\SS$.

\bigskip

\bigskip


\begin{center}

{\Large\bf Part A: A quasianalyticity property}

\end{center}


\section{Various notions of quasianalyticity}

\label{secVarQA}


In this article we call ``non-trivial path'' the image of any non-constant
continuous map from~$[0,1]$ to $\PP=\C\cup\{\infty\}$, ``Jordan arc'' the image
of a continuous injective map from~$[0,1]$ or~$(0,1)$ to~$\PP$, and ``Jordan
curve'' the image of a continuous injective map from $\R/\Z$ to~$\PP$.
The one-dimensional Hausdorff outer measure in~$\C$ will be
denoted~$\gH^1$; we extend it to~$\PP$ by setting $\gH^1(A) =
\gH^1(A\setminus\{\infty\})$ for any $A\subset\PP$ (in fact, what will
matter for us will not be the precise value of~$\gH^1(A)$, but whether
it is positive or not).

The following definition is inspired by T.~Carleman \cite[p.2]{Car} and
A.~Beurling \cite[p.396]{Beur} (see also \cite[p.275]{Kolog}).
\begin{Def}    \label{defQA}
  Let $K'\subset K$ be subsets of~$\PP$ and $E$ be a linear space of
  functions on~$K$ with values in a complex Banach space.
  \begin{itemize}
  \item A subset~$\ga$ of~$K$ is said to be a {\em uniqueness set} for~$E$
    if the only function of~$E$ vanishing identically on~$\ga$ is the
    function $f\equiv0$.
%
%
  \item We say that~$E$ is {\em $\gH^1$-quasianalytic relatively to~$K'$}
    if any subset of $K'$ of positive $\gH^1$-measure is a
    uniqueness set for~$E$.
  \end{itemize}
\end{Def}

Since every non-trivial path has positive $\gH^1$-measure (see \eg\
\cite[p.29]{Falco}), $\gH^1$-quasianalyticity relatively to~$K'$ implies that
any non-trivial path contained in~$K'$ is a uniqueness set, a property which
could be termed {\em pathwise quasianalyticity relatively to~$K'$}.
%

As is well-known, if $\La$ is a Jordan arc, then $\gH^1(\La)$ coincides with its
length,~$|\La|$.
When this number is finite (\ie\ when the arc avoids~$\infty$ and is rectifiable
with respect to the usual distance of~$\C$), for any~$U$ open in~$\La$ one can
define $\length(U) = \sum|U_j|$, where the $U_j$'s are its connected components;
setting, for any subset~$A$ of~$\La$, $\length_\La(A) = \inf\{ \length(U)\,;\; 
\text{$U$ open in~$\La$, $A\subset U$} \}$,
we then have\footnote{
Indeed, the identity~(\ref{eqarcHausd}) holds for all open subsets of~$\La$ and
$\gH^1$ is a Borel-regular measure on~$\C$ \cite[\S\S2.10.2, 2.10.13]{Fed}; the
measure it induces on~$\La$ is Borel-regular and finite, thus
  $\gH^1(A) = \inf\{ \gH^1(U)\,;\; \text{$U$ open in~$\La$, $A\subset
    U$} \}$ for each $A\subset\La$
\cite[\S\S2.2.2--2.2.3]{Fed}.
}
\begin{equation}   \label{eqarcHausd}
\gH^1(A) = \length_\La(A).
\end{equation}

\begin{Def}    \label{defHadaQA}
  Let~$q_0$ be a non-isolated point of $K\subset\PP$ and~$E$ be a
  linear space of functions on~$K$ with values in a complex Banach space~$B$, such
  that each function of~$E$ admits an asymptotic expansion at~$q_0$.
  We say that~$E$ is {\em quasianalytic at~$q_0$} if the only function with zero
asymptotic expansion at~$q_0$ is the function $f\equiv0$.
\end{Def}

Recall that a function $f\colon K\to B$ is said to admit an asymptotic expansion at~$q_0$ if there
exists a sequence $(a_n)_{n\in\N}$ of elements of~$B$ such that,
for every $N\in\N$,
$(q-q_0)^{-N}\bigl( f(q) - \sum_{n=0}^{N} a_n (q-q_0)^n \bigr)$
tends to~$0$ as $q\to q_0$ with the constraint $q\in K$. 
The sequence of coefficients is then unique:
\begin{equation}	\label{eqdefAE}
a_N = \lim_{q\to q_0} (q-q_0)^{-N}\left( f(q) - \sum_{n=0}^{N-1} a_n (q-q_0)^n
\right),
\qquad N\in\N.
\end{equation}
This hypothesis is met if $f$ is analytic at~$q_0$, but also when $K$ is closed
and $f$ is Whitney-differentiable infinitely many times in the complex sense (\ie\
$\gC^\infty$-holom\-orphic, see below) on~$K$.

According to Definition~\ref{defHadaQA}, quasianalyticity at~$q_0$ means that
the functions are determined by the coefficients~$a_n$ of their asymptotic
expansions.
Observe that this implies that any set $\ga\subset K$ of which~$q_0$ is a limit
point is a uniqueness set for~$E$. Indeed, formula~(\ref{eqdefAE}) shows that
the coefficients of the asymptotic expansion of a function are inductively
determined by its restriction to~$\ga$.

The usual notion of quasianalyticity (in the sense of Hadamard) is
quasianalyticity at every point (see \eg\ \cite{Th}).
The latter property is a priori stronger than $\gH^1$-quasianalyticity
relatively to~$K$ (because any set~$\ga$ of positive $\gH^1$-measure
has a limit point in it).\footnote{
Notice however that, for the Denjoy-Carleman classes of an interval of the real
line, the notions of $\gH^1$-quasianalyticity (or pathwise quasianalyticity)
relatively to this interval and Hadamard quasianalyticity coincide---see \eg\
\cite[p.9]{Car}---but this has to do with the one-dimensional
character of the interval, whereas we shall rather be interested in compact
subsets~$K$ of~$\PP$ not contained in any line.
}

If the interior of~$K$, henceforth denoted by~$\ovint{K}{1.45}$, has several
connected components, the pathwise or $\gH^1$-quasianalyticity of~$E$ relatively
to~$\ovint{K}{1.45}$
%
%
%
imply a form of {\em coherence}: if two functions of~$E$
coincide in one of the connected components of~$\ovint{K}{1.45}$, then they
coincide everywhere;
given a function of~$E$, one may also think of its restriction to any of the
components as of the ``pseudocontinuation'' or ``generalized analytic
continuation'' of its restriction to one of them, even though analytic
continuation may be impossible (compare with \cite[pp.18,49]{RS}).
Similar remarks apply when $E$ is quasianalytic at the points of~$\ovint{K}{1.45}$.


\section{$\gC^1$-holomorphic functions and monogenic functions}
%
\label{secMonFcn}
%
%
As in \cite{MS1}, we are interested in functions which are
$\gC^1$-holomorphic on compact sets of the Riemann sphere, \ie\ these
functions are Whitney-differentiable and satisfy the Cauchy-Riemann
equations;
equivalently, for a compact set~$K$ in~$\PP$ and a complex Banach
space~$B$, we say that $f \colon K \to B$ is $\gC^1$-holomorphic if it
is continuous and there exists a continuous $f\pp1 \colon K \to B$ such
that: for all $q\in K$ and $\eps>0$, there exists $\de>0$ with
$$
q_1, q_2 \in K, \ens |q_1-q|, |q_2-q| \leq \de 
\ens \Rightarrow \ens
\norm{ f(q_2) - f(q_1) - (q_2-q_1) f\pp1(q) } \leq \eps|q_2-q_1|
$$
(using inversion if $q=\infty$, as usual).
We then use the notation $f \in \Chol(K,B)$; the linear space of functions we
get can be made a complex Banach space by choosing appropriately a norm
$\norm{\,.\,}_{\Chol(K,B)}$ (see \cite[\S2.1]{MS1}). 
The definition of $\gC^\infty$-holomorphic functions on~$K$ is in the same vein
(op.cit.; \cite{Ri}).

For the moment we impose no restriction on the compact sets~$K$ we consider,
but in Section~\ref{secStatRes} we shall restrict ourselves to very specific
ones (see Figure~\ref{figCurves}).

\begin{Def}
Suppose $(K_j)_{j\in\N}$ is a monotonic non-decreasing sequence of compact subsets of~$\PP$
and $(B_j)_{j\in\N}$ is a monotonic non-decreasing sequence of complex Banach
spaces with continuous injections $B_j \hookrightarrow B_{j+1}$.
The corresponding space of {\em monogenic functions} is the Fr\'echet space obtained
as the projective limit of Banach spaces
\begin{multline*}
\gM\bigl( (K_j), (B_j) \bigr) = \varprojlim \cA_J, \\[1ex]
\cA_J =  \bigcap_{0\le j\le J}\Chol(K_j,B_j), \qquad
\norm{f}_{\cA_J} = \max_{0\le j\le J} \norm{f_{|K_j}}_{\Chol(K_j,B_j)}.
\end{multline*}
\end{Def}

Indeed, the $\cA_J$'s with the continuous injections $\cA_J \hookrightarrow
\cA_{J-1}$ give rise to a projective system
and the projective limit $\gM\bigl( (K_j), (B_j) \bigr)$ is a complete
topological vector space for the family of semi-norms
$\bigl(\norm{\,.\,}_{\cA_J}\bigr)_{J\ge0}$.
As a set, $\gM\bigl( (K_j), (B_j)
\bigr)$ consists of all the functions which are defined in
$\gF = \bigcup_{j\in\N} K_j$
and such that, for every $j\in\N$, the restriction $f_{|K_j}$ belongs to
$\Chol(K_j,B_j)$
(this space may depend on the sequence~$(K_j)_{j\in\N}$ rather than on~$\gF$ only).

In \cite{He} or \cite{MS1}, the definition is given with a fixed Banach space
$B=B_j$ for all $j\in\N$, in which case $\cA_J = \Chol(K_J,B)$.
Typically, $B$ is the Hardy space $H^\infty(\D_r)$ consisting of bounded
holomorphic functions in a disk $\D_r = \{ |z|<r \}$.
When applying these ideas to the linear small divisor problem described in
Section~\ref{secIntroSDpb}, the drawback of keeping~$B_j$ constant as in \cite{MS1}
is that the optimal arithmetical condition (see~(\ref{eqLinOpt}) below)
cannot be reached: $\gF$ is smaller than it could be.
Similarly, in the Siegel problem, capturing all the points of the unit circle
which satisfy the Bruno condition~(\ref{eqBrunoClassic}) requires to consider a
sequence of decreasing disks.
We shall thus take $B_j = H^\infty(\D_{r_j})$ with $r_j \downarrow 0$ in these applications.

In all these cases,
each~$K_j$ will be a compact arcwise connected subset of~$\PP$,
which intersects the unit circle along a Cantor set avoiding the roots of unity,
the interior of which has two connected components, one inside and the other
outside the unit circle, while
\begin{equation}  \label{eqintCclosC}
\bigcup_{j\in\N} \ovint{K}{1.45}_j = \{\, |q|<1 \,\} \cup \{\, |q|>1 \,\}
\ens\subset\ens \gF  = \bigcup_{j\in\N} K_j \ens\subset\ens
\ov \gF = \PP
\end{equation}
(both inclusions will be strict).

Obviously, for any compact~$K$ and Banach space~$B$, we have the inclusion
\begin{equation}	\label{eqdefgO}
\Chol(K,B) \subset \gO(K,B) = \{\, f \colon K \to B \;
\text{continuous in~$K$ and holomorphic in~$\ovint{K}{1.45}$}
\,\}.
\end{equation}
In the results of the following section, it is in fact~$\gO(K,B)$ itself which
will be proved to enjoy quasianalyticity properties in certain circumstances,
and this will imply similar properties for the smaller space $\Chol(K,B)$.
We thus find it worthwhile to mention an inclusion which goes in the reverse
direction. Recall that the inner boundary of~$K$ is defined as $\pa K\setminus\bigcup\pa
U_\ell$ where the $U_\ell$'s denote the connected components of~$\PP\setminus K$.
\begin{lemma}
Assume that the inner boundary of~$K$ is contained in an analytic curve and that
the boundary of each connected component~$U_\ell$ of~$\PP\setminus K$ is a union of
rectifiable Jordan curves.
Suppose that $K$ contains a compact~$\ti K$ such that any two close enough
points $q$,~$q'$ of~$\ti K$ can be joined by a rectifiable path of length $\le
c|q-q'|$ inside~$\ti K$, and
\begin{equation}	\label{ineqseriescc}
q\in \ti K \ens\Rightarrow\ens
\sum_\ell \int_{\pa U_\ell} \frac{|\dd\ze|}{|\ze-q|^3} \le C,
\end{equation}
where $c$ and~$C$ are positive constants. Then  $\gO(K,B) \subset \Chol(\ti K,B)$.
\end{lemma}

\begin{proof}
By Melnikov's theorem \cite[p.112]{Za}, the assumption implies that any
$f\in\gO(K,B)$ is the uniform limit of a sequence of rational functions $r_k$
with poles off~$K$.
Given $q\in\ti K$, the function $\ze\mapsto f(\ze)$ and $\ze\mapsto |\ze-q|$ are
bounded on~$K$, thus there exists $\ka>0$ such that $|\ze-q|\ii\le
\ka|\ze-q|^{-3}$ and $|\ze-q|^{-2}\le \ka|\ze-q|^{-3}$ and we can set
$$
f\pp0(q) = \frac{1}{2\pi\I} \sum_\ell \int_{\pa U_\ell} \frac{f(\ze)}{\ze-q}\,\dd\ze,
\quad
f\pp1(q) = \frac{1}{2\pi\I} \sum_\ell \int_{\pa U_\ell} \frac{f(\ze)}{(\ze-q)^2}\,\dd\ze.
$$
With a suitable orientation of the $\pa U_\ell$'s, applying the Cauchy theorem
to the rational functions~$r_k$ and passing to the limit, we see that
$f\pp0(q)=f(q)$.

Take now $q,q'\in \ti K$ close enough one to the other, with a rectifiable path~$\ga$ joining
them inside~$\ti K$. It will be sufficient to show that
$A := \norm{f(q')-f(q)-(q'-q)f\pp1(q)}$ is $O\bigl(\bigl(\length(\ga)\bigr)^2\bigr)$.
We have 
$$
f(q')-f(q)-(q'-q)f\pp1(q) = 
\frac{1}{2\pi\I} \sum_\ell \int_{\pa U_\ell} f(\ze) \cR(q,q',\ze)\,\dd\ze,
$$
where 
\[
\cR(q,q',\ze) = \frac{1}{\ze-q'} - \frac{1}{\ze-q} -
\frac{q'-q}{(\ze-q)^2}
\]
can also be written $\int_\ga
\frac{2(q'-q_1)}{(\ze-q_1)^3}\,\dd q_1$ (Taylor formula with integral remainder).
By Fubini's theorem, we get
$$A \le \frac{C}{\pi}\max|f| \int_\ga |q'-q_1|\,|\dd q_1|$$
and the conclusion follows.
\end{proof}

A similar idea is used in Section~2.5 of \cite{MS1} (see also Remark~2.1 there),
where specific compact sets $K^*\subset K$ are defined and satisfy conditions
stronger than~(\ref{ineqseriescc}) which imply $\Chol(K,B) \subset
\Cinfhol(K^*,B)$
(the inner boundaries of the compact sets used in the application to small divisor
problems are contained in the unit circle).


\section{A quasianalyticity result for $\gO(K,B)$ and monogenic functions}
\label{secStatRes}


The compact sets we are interested in are defined as follows:
\begin{Def}
We say that $\bigl(\Ga\IN,\Ga\EX\bigr)$ is a {\em nested pair} if~$\Ga\IN$ and $\Ga\EX$ are
Jordan curves contained in~$\C$ such that
\begin{itemize}
\item $\Ga\IN$ is contained in the closure of the connected component
of~$\PP\setminus\Ga\EX$ which does not contain~$\infty$,
\item $\Ga\IN$ and~$\Ga\EX$ are rectifiable and $\gH^1(\Ga\IN\cap\Ga\EX)>0$.
\end{itemize}
We then define~$K\EX$ to be the closure of the connected component of $\PP\setminus\Ga\EX$
which contains~$\infty$ (delimited by the ``external''
curve~$\Ga\EX$),
$K\IN$ to be the closure of the connected component of
$\PP\setminus\Ga\IN$ which does not contain~$\infty$ (delimited by the ``internal''
curve~$\Ga\IN$),
and 
\begin{equation}  \label{eqdefK}
K\bigl(\Ga\IN,\Ga\EX\bigr) = K\IN \cup K\EX,
\quad \cI\bigl(\Ga\IN,\Ga\EX\bigr) = \Ga\IN\cap\Ga\EX.
\end{equation}
\end{Def}

In practice, in the applications considered in Part~B, we shall have furthermore
\begin{equation}  \label{eqGaINGaEX}
\Ga\IN \subset \ov\D\setminus\{0\},
\quad \Ga\EX \subset \ov\E\setminus\{\infty\}, 
\quad \Ga\IN \cap \SS = \Ga\EX \cap \SS, 
\end{equation}
where
$$
\D = \{\, q\in\C \mid |q|<1 \,\}, \quad
\E = \{\, q\in\PP \mid |q|>1 \,\}, \quad
\SS = \{\, q\in\C \mid |q|=1 \,\},
$$
and the set $\cI\bigl(\Ga\IN,\Ga\EX\bigr)\subset\SS$ will be defined by an arithmetical
condition which gives it positive Lebesgue measure on the unit circle.
Then $\cI\bigl(\Ga\IN,\Ga\EX\bigr)$ coincides with $K\IN \cap \SS = \Ga\IN \cap
\SS$ and with $K\EX \cap \SS = \Ga\EX \cap \SS$
(see Figure~\ref{figCurves}).


\begin{figure}

\begin{center}

\psfrag{OR}{$0$}

\psfrag{KI}{$K\IN$} \psfrag{KE}{$K\EX$}
\psfrag{GI}{$\Ga\IN$} \psfrag{GE}{$\Ga\EX$}

\epsfig{file=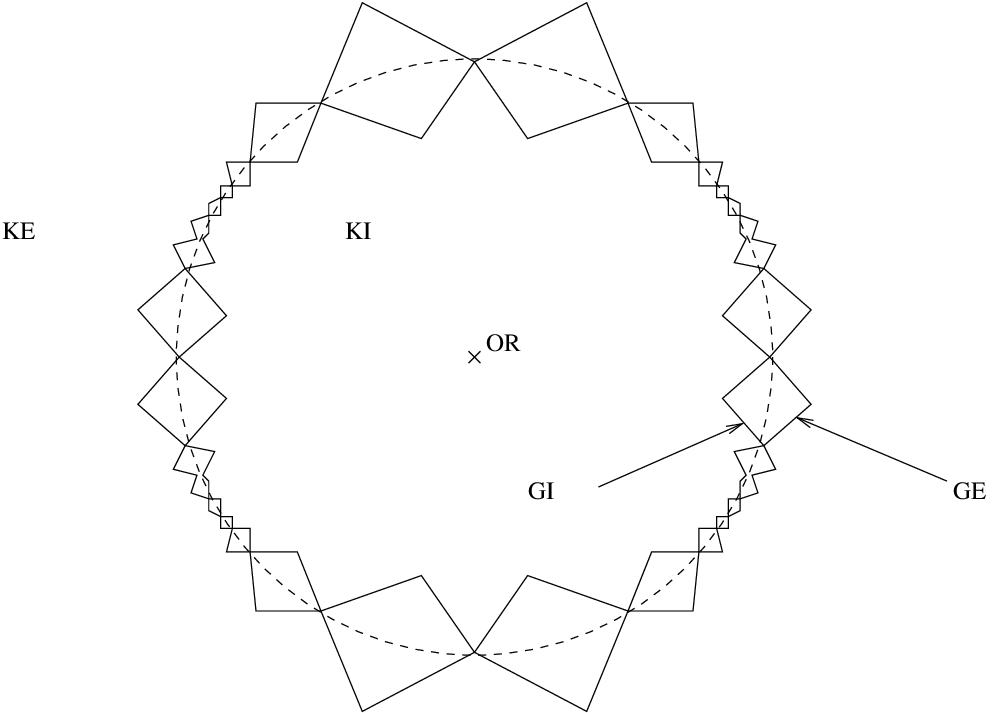,height=2.6in,angle = 0}

\end{center}

\caption{\label{figCurves} $K=K\bigl(\Ga\IN,\Ga\EX\bigr)$ is the union of the
sets $K\IN$ and $K\EX$ delimited by the internal and external curves $\Ga\IN$
and $\Ga\EX$.}

\end{figure}

\begin{thmA}
   Let $\bigl(\Ga\IN,\Ga\EX\bigr)$ be a nested pair and~$B$ a complex Banach
space.
   Let $K=K\bigl(\Ga\IN,\Ga\EX\bigr)$.
  Then $\gO(K,B)$ is $\gH^1$-quasianalytic relatively to~$K$ and it is also
  quasianalytic at every point of~$\ovint{K}{1.45}$.
\end{thmA}

Being a smaller space, $\Chol(K,B)$ inherits these quasianalyticity properties.
Observe that, by construction, $\ovint{K}{1.45}$ has two connected components,
$\ovint{K}{1.45}\str{1.55}\IN$ and $\ovint{K}{1.45}\str{1.55}\EX$ (respectively
contained in~$\D$ and~$\E$ when~(\ref{eqGaINGaEX}) is fulfilled); 
our functions thus enjoy the aforementioned coherence property, while examples
in \cite{MS1} show that the unit circle may be a barrier for the ordinary
analytic continuation.

\begin{corA}
  Let $(B_j)_{j\in\N}$ denote a monotonic non-decreasing sequence of complex
  Banach spaces with continuous injections.
  Assume that $\bigl( \Ga\IN_j,\Ga\EX_j \bigr)_{j\in\N}$ is a sequence of nested
  pairs
  such that the sequence of compact
  sets defined by $K_j = K\bigl( \Ga\IN_j, \Ga\EX_j \bigr)$ is
  monotonic non-decreasing.

  Then the space of monogenic functions $\gM\bigl( (K_j), (B_j) \bigr)$ is
  $\gH^1$-quasianalytic relatively to $\bigcup K_j$ and
  it is also quasianalytic at every point of $\bigcup \ovint{K}{1.45}_j$.
\end{corA}

Observe that, under the assumption~(\ref{eqintCclosC}) (which will hold in the
applications of Part~B), the functions
of $\gM\bigl( (K_j), (B_j) \bigr)$ are holomorphic both in~$\D$ and
in~$\E$. They enjoy the aforementioned coherence property: pseudocontinuation is
possible through the unit circle.

\begin{rem}	\label{remNonQA}
According to \cite{MS1}, for the kind of compact sets
$K=K\bigl(\Ga\IN,\Ga\EX\bigr)$ which appear in the small divisor problems of Part~B, it is
possible to reinforce the arithmetical condition which defines $\cI = K\cap\SS$
so as to define $\cI^* = K^*\cap\SS \subset \cI$, where $K^* = K\bigl( \Ga\IN_*, \Ga\EX_*
\bigr) \subset K$ with a new nested pair $(\Ga\IN_*,\Ga\EX_*)$,
in such a way that the functions of $\Chol(K,B)$ admit asymptotic expansions of
Gevrey type at the points of~$\cI^*$.
Besides, $\Chol(K,B) \subset \Cinfhol(K^*,B)$, as was alluded to at the end of Section~\ref{secMonFcn}.
One can then raise the question of the (Hadamard) quasianalyticity at the points
of~$K^*$, but this is more difficult. 

For instance, the points of~$\SS$ which satisfy the so-called ``constant-type''
Diophantine condition belong to~$\cI^*$, and it is shown in \cite[\S3.3]{MS1} that
$\Chol(K,B)$ is not contained in any of the classical Carleman classes
quasianalytic at these points---it is in fact the solution itself of the linear
small divisor problem that does not belong to these quasianalytic classes.

In \cite{He}, Herman alludes to Borel's studies to determine conditions on the
$K_j$'s which ensure the quasianalyticity of the monogenic functions at least at
certain points, but they are not fulfilled here (see \cite{Wk} and \cite[Remark~2.4]{MS1}).

\end{rem}


\vspace{8pt}

\noindent {\em Proof of Theorem~A.}\;
%
%
Let $f\in \gO(K,B)$; we must infer $f\equiv0$ from its vanishing on
certain subsets of~$K$ or from the vanishing of all its derivatives at a given
point of~$\ovint{K}{1.45}$.
Without loss of generality, we can take $B = \C$ (because the dual
of~$B$ separates the points of~$B$ and 
$\ell\circ f \in \gO(K,\C)$ for every $\ell\in B^*$).
Let $U\IN$ denote the connected component of $\PP\setminus\Ga\IN$
which contains~$0$ (which is nothing but the interior of~$K\IN$) and
$V\EX$ the connected component of $\PP\setminus\Ga\EX$ which
contains~$\infty$ (the interior of~$K\EX$).


\medskip

\noindent{\bf (a)} 
We first prove that $f_{|\cI} \equiv0 \;\Rightarrow\; f\equiv0$.
The key argument comes from Koosis's proof of Privalov's uniqueness
theorem \cite{Ko}.

By a theorem of Carath\'eodory, any conformal representation $\ph\IN
\colon \D \to U\IN$ extends to a homeomorphism $\Phi \colon \ov\D \to
K\IN$.
Assuming $f_{|\cI} \equiv0$, we thus have a function $F\IN =
f\circ\Phi \in \gO(\ov\D,\C)$ which vanishes identically on $\cJ =
\Phi\ii(\cI) \subset \SS$.
According to a theorem of F.\ and M.~Riesz, the image by~$\Phi$ of a
subset of zero Lebesgue measure of~$\SS$ has zero arc-length
on~$\Ga\IN$ (see \cite[p.54]{Ko});
now $\cI \subset \Ga\IN$ and, by virtue of~(\ref{eqarcHausd}),
$\length_{\Ga\IN}(\cI) = \gH^1(\cI)> 0$, hence the
Lebesgue measure of~$\cJ$ is positive.
It follows, by a uniqueness theorem for $H^1$ functions (see
\cite[p.57]{Ko}), that $F\IN\equiv0$.

Hence $f$ vanishes identically on~$K\IN$. A similar argument with a
conformal mapping $\ph\EX \colon \D \to V\EX$ yields $f\equiv0$
on~$K\EX$.


\medskip

\noindent{\bf (b)} 
Suppose now that all the derivatives of~$f$ vanish at a point $q_* \in
\ovint{K}{1.45}$. The principle of analytic continuation yields $f\equiv0$ on
the connected component of~$q_*$ in~$\ovint{K}{1.45}$, which is either~$U\IN$
or~$V\EX$, and then, by continuity, $f\equiv0$ on~$K\IN$ or~$K\EX$ accordingly.
In particular, $f\equiv0$ on~$\cI$ and, by~(a), we
get $f\equiv0$ on the whole of~$K$.


\medskip

\noindent{\bf (c)} 
Suppose finally that $\ga$ is a subset of~$K$ with $\gH^1(\ga)>0$.
By subadditivity, at least one of the sets $\ga\cap \ovint{K}{1.45}$,
$\ga\cap \Ga\IN$ or $\ga\cap \Ga\EX$ must have positive
$\gH^1$-measure.

\medskip

\noindent -- If the first set has positive $\gH^1$-measure, then it
has an accumulation point and we get $f\equiv0$ by virtue
of~(b)
(cf.~(\ref{eqdefAE}) and the observation in the paragraph
following it).

\medskip

\noindent -- If one of the last two ones, say $\ga\cap \Ga\IN$, has
positive $\gH^1$-measure, then (\ref{eqarcHausd}) implies that
$\length_{\Ga\IN}(\ga\cap\Ga\IN)>0$ and we can argue as
in~(a): 
we take any conformal representation $\ph\IN\colon\D\to U\IN$, extend
it to a homeomorphism $\Phi\colon\ov\D\to K\IN$; we have $F\IN =
f\circ\Phi$ holomorphic in~$\D$, continuous in~$\ov\D$, vanishing on
$\Phi\ii(\ga\cap\Ga\IN) \subset \SS$ which has positive Lebesgue
measure (still by F.\ and M.~Riesz's theorem); thus~$F\IN$ vanishes
identically and so does $f_{|K\IN}$.
In particular $f\equiv0$ on~$\cI$ and, by the result obtained
in~(a), $f\equiv0$ on the whole of~$K$.
\qed


\vspace{8pt}

\noindent {\em Proof of Corollary~A.}\; 
%
%
%
We show in fact slightly more: the space $\bigcap \gO(K_j,B_j)$, which is
usually larger than $\gM\bigl( (K_j), (B_j) \bigr)$, is itself
$\gH^1$-quasianalytic relatively to $\gF=\bigcup K_j$ and quasianalytic at the
points of 
\[
\gF' = \bigcup
\ovint{K}{1.45}_{j} \subset \ovint{\gF}{1.45}.
\]

Let $f \in \bigcap \gO(K_j,B_j)$ (this function is thus defined in~$\gF$),
let $\ga \subset \gF$ with $\gH^1(\ga)>0$ and let $q_* \in \gF'$.
By subadditivity of~$\gH^1$, we have $\gH^1(\ga\cap K_j)>0$ for $j$ large
enough.
Also, $q_*\in\ovint{K}{1.45}_{j}$ for $j$ large enough.

Assuming that $f$ vanishes on~$\ga$, Theorem~A thus yields $f\equiv0$ on~$K_j$
for all $j$ large enough, hence on~$\gF$.
The same is true if we assume that the derivatives of~$f$ vanish at~$q_*$.
\qed


\medskip

\begin{rem} {\em (Relation with Privalov's uniqueness theorem.)}\;
\label{remPriv}
%
%
Privalov's theorem asserts that a function holomorphic in~$\D$
with zero as non-tangential limit at every point of a subset of
positive Lebesgue measure of~$\SS$ must vanish identically.

Assume that the hypotheses of Corollary~A are satisfied and let $\gF=\bigcup
K_j$. Assume moreover that~(\ref{eqintCclosC}) holds and that,
for any $\la = \e{2\pi\I x} \in \gF\cap\SS$ and $c,d>0$
such that $d<1$, there exists $j\in\N$ such that the region
\[ \{\, q = r\,\e{2\pi\I\th} \mid c |\th-x|\le |r-1| \le d \,\} \]
is contained in~$K_j$.

These extra assumptions will be met in the small divisor problems of
Part~B. They imply that a function $f\in\gM\bigl( (K_j), (B_j)
\bigr)$ is holomorphic in $\D\cup\E$ and that, at each point $\la\in \gF\cap\SS$,
it admits~$f(\la)$ as non-tangential limit.

In this case, if $f$ vanishes on a set $\ga\subset \gF\cap\SS$ of positive
Lebesgue measure, one can deduce $f\equiv0$ on~$\D$ directly from Privalov's
theorem, and also $f\equiv0$ on~$\E$ by the same theorem (using inversion),
hence $f$ vanishes on the whole of~$\gF$ by continuity.

\end{rem}


\bigskip

\bigskip


\begin{center}

{\Large\bf Part B: Applications to linear and non-linear small divisor problems}

\end{center}



\section{Introduction to small divisor problems}	\label{secIntroSDpb}
%
%

We now fix the notations for three small divisor problems, to which the results of
Part~A will be applied in the subsequent sections.
They are, by order of increasing complexity, Problem~\tn{L},
Problem~\tn{S} and Problem~\tn{C}.
We begin by presenting the second one, which is the so-called Siegel problem:
%
%


\begin{pbS}
Let $G(z) = z + g(z)$ with $g(z)= \sum_{k\ge2} g_k z^k\in z^2\C\{z\}$. 
Study the solution 
$h(z) = z + \sum_{k\ge2} h_k z^k$ of the conjugacy equation
\begin{equation}	\label{eqConj}
h(qz) = q G\bigl( h(z) \bigr)
\end{equation}
as a function of the parameter $q\in\C$.
\end{pbS}

Equation~(\ref{eqConj}) describes indeed the conjugacy between the germ 
$z\mapsto 
q\cdot\bigl( z + g(z) \bigr)$ and its linear part $z\mapsto qz$;
the parameter~$q$ is called the multiplier.
It is well known that, when $q$ is not a root of unity, (\ref{eqConj}) has a unique
formal solution~$h$ tangent to the identity.
The power series~$h(z)$ is always convergent when $q\in\C^*\setminus\SS$,
whereas an arithmetical condition is needed when $q=\ee^{2\pi\I \al}\in\SS$: the
so-called Bruno condition \cite{Bruno,Yoccoz}, which reads
\begin{equation}	\label{eqBrunoClassic}
\sum_{k=0}^\infty \frac{\log m_{k+1}(\al)}{m_k(\al)} < \infty,
\end{equation}
where $\bigl(m_k(\al)\bigr)_{k\ge0}$ denotes the sequence of the denominators of
the convergents of the irrational real number~$\al$ (see Appendix~\ref{appCF}).
The set 
\begin{equation}	\label{eqBrunoNumb}
\cI^{\tn{S}} = \{\, q=\ee^{2\pi\I \al} \mid
\al\in\R\setminus\Q \;\text{satisfies (\ref{eqBrunoClassic})} \,\}
\end{equation}
has full measure in~$\SS$.

Since the solution of~(\ref{eqConj}) depends on~$q$, we shall denote it by
$h(q,z)$ instead of~$h(z)$.
Its coefficients are uniquely determined by induction and are rational functions
of~$q$: with the convention $h_1=1$, the recurrence formulas are
\begin{equation}\label{eq:recurrence}
h_k = \frac{1}{q^{k-1}-1} \, \sum_{j=2}^{k} \, g_j
\sum_{\substack{k_1,\ldots,k_j\ge1 \\ k_1+\cdots+k_j=k}}
h_{k_1}\cdots h_{k_j}, \qquad k\ge2.
\end{equation}
It is easy to see that $q\mapsto h(q,\,.\,)$ is analytic at each point
$q_0\notin\SS$ (with values in a Banach space of holomorphic functions of~$z$
which depends on~$q_0$), 
including the extreme cases $q_0 = 0$ and $q_0 = \infty$ 
(for which $h_{|q=0}$ is the functional inverse of~$G$ and $h_{|q=\infty}$ is
reduced to the identity),
\ie\ it is natural to let~$q$ vary in~$\PP$ in Problem~\tn{S}.

From the point of view of analyticity \wrt~$q$, we thus get two distinct
holomorphic functions $h_{|\D}$ and~$h_{|\E}$, but analytic continuation
through~$\SS$ is impossible, at least in the classical sense, because of the small divisors
$q^{k-1}-1$ (each root of unity acts as a resonance, being a pole for infinitely
many $h_k$'s). 
Still, whenever (\ref{eqBrunoClassic}) is satisfied, $h(\ee^{2\pi\I \al},\,.\,)$
is the limit of~$h(q,\,.\,)$ as $q\to \ee^{2\pi\I \al}$ non-tangentially (cf.\
Section~\ref{remPriv}), as is mentioned in \cite[\S2.1]{BMS}.
The point here is to go farther than this continuity property by using the
results of Part~A.

\medskip


Linearizing the conjugacy equation~(\ref{eqConj}) written as
$h(qz) - qh (z) = q g\bigl(h(z)\bigr)$
leads to

\begin{pbL}
Let $g(z)= \sum_{k\ge2} g_k z^k\in z^2\C\{z\}$. 
Study the solution 
\begin{equation}	\label{eqSolCoho}
h(z) = h_{g}(q,z) = z + \sum_{k\ge 2}g_{k}\frac{z^k}{q^{k-1}-1}
\end{equation}
of the ``cohomological equation''
\begin{equation}  \label{eq:cohomological} 
h(qz) - q h(z) = q g(z)
\end{equation}
as a function of the parameter $q\in\PP$.
\end{pbL}

The series~(\ref{eqSolCoho}) is uniformly convergent in each compact subset of
$(\PP\setminus\SS)\times\D_{R}$, where 
$$\D_R = \ao z\in\C\mid |z|<R \af$$ is the disk of
convergence of~$g(z)$.
Thus, also in this case we get two holomorphic functions, one for $q\in\D$ and
one for $q\in\E$.
Also in this case there are non-tangential limits at certain points of the unit
circle; this time, the optimal condition for the convergence of the power series
in~$z$ when $q=\ee^{2\pi\I \al}\in\SS$ is simply:
\begin{equation}	\label{eqLinOpt}
\sup \left\{\frac{\log m_{k+1}(\al)}{m_k(\al)},\; {k\ge0} \right\} < \infty
%
\end{equation}
(see \cite[\S A.2]{MS1}); in a way similar to~(\ref{eqBrunoNumb}), this defines 
a full-measure subset~$\cI^{\tn{L}}$ of~$\SS$, which is a countable union
of nowhere dense closed sets, while its complement in~$\SS$ is a dense
$G_\de$-set with zero $s$-dimensional Hausdorff measure for all $s>0$.

\medskip


The last problem we consider is again non-linear; it can be viewed as a
complexification of the problem of local conjugacy of a circle maps:

\begin{pbC}
Let $G_{\al,\eps}(\th) = \th + \al + \eps g(\th)$ with
$g(\th) = \sum_{k\in\Z^*} g_k \,\ee^{2\pi\I k\th}$ holomorphic in the annulus
$S_R = \ao \th\in\C/\Z \mid |\IM\th|<R \af$.
Study the solution $(\be,h)$, where $\be\in\C$ and $h:\th\mapsto\th+u(\th)$ with
$u$ a $1$-periodic holomorphic function of zero mean-value, of the conjugacy equation
\begin{equation}	\label{eqComplexCircle}
G_{\al,\eps}\bigl(h(\th)\bigr) - \be = h(\th+\al)
\end{equation}
as a function of the parameters $q=\ee^{2\pi\I\al}\in\PP$ and~$\eps\in\C$ for small~$|\eps|$.
\end{pbC}

Equation~(\ref{eqComplexCircle}) describes the conjugacy between the map
$G_{\al,\eps}-\be=G_{\al-\be,\eps}$ and the rigid rotation $G_{\al,0}:\th\mapsto\th+\al$.
Observe that the correction~$\be$ is needed and it is impossible to impose a
priori its value since the ``rotation number'' $\rho(G_{\al,\eps})$ need not
coincide with~$\al$ for nonzero~$\eps$; in fact, $\be$ is implicitly {\em
determined} (modulo~$1$) by the equation
\begin{equation}	\label{eqrho}
\rho(G_{\al-\be,\eps}) = \al
\end{equation}
as a function of~$\al$ and~$\eps$ which is $1$-periodic in~$\al$.
Of course, to speak of complex rotation number, we need a generalization \wrt\
the classical theory of circle diffeomorphisms (in which $g(\th)$ is assumed to
be real for real values of~$\th$ and only real values of~$\al$ and~$\eps$ are
considered)---see \cite{Ri} for a geometric insight on this generalization.

Instead of solving first equation~(\ref{eqrho}) and then the conjugacy
equation~(\ref{eqComplexCircle}), it is possible to obtain directly the pair
$(\be,h)$ by rewriting the conjugacy equation as 
$h(\th+\al)-h(\th)-\al+\be = \eps g\bigl(h(\th)\bigr)$
and defining the operator 
\begin{equation}	\label{eqdefEq}
E_q\,:\; v(\th) = \sum_{k\in\Z} v_k\,\ee^{2\pi\I k\th} \;\mapsto\;
E_q v(\th) = \sum_{k\in\Z^*} \frac{v_k}{q^k-1}\ee^{2\pi\I k\th}
\end{equation}
for $q=\ee^{2\pi\I\al}$.
Indeed, denoting the mean-value of a $1$-periodic function by
$\ld\,.\,\rd$, Problem~\tn{C} is then equivalent to
\begin{equation}	\label{eqdefbehv}
\be = \ld v \rd, \quad h(\th) = \th + E_q v(\th),
\end{equation}
with $v$ solution of the fixed-point equation
\begin{equation}	\label{eqfixedv}
v(\th) = \eps g\bigl(\th+E_q v(\th)\bigr).
\end{equation}
In Section~\ref{secSDpbNLC} we shall see that, for $q\in\PP\setminus\SS$, $E_q$
defines a bounded operator on the Banach space $\gO(\ov S_{R/2},\C)$ of the
holomorphic functions in the annulus~$S_{R/2}$ which are continuous on its
closure; this will allow us to prove, for $|\eps|$ small enough, the existence
of a unique solution~$(\be,h)$ close to~$(0,\ID)$, which depends holomorphically
on~$q\in\PP\setminus\SS$ and~$\eps$.

On the other hand, \cite{He} deals with the regularity in~$q$ of this solution
in subsets of $\ao q=\ee^{2\pi\I\al} \mid \al\in\C/\Z, \; |\IM\al|\le
R/100 \af$ defined with the help of an arithmetical condition:
a constant $\tau\in(0,1)$ being fixed once for all, these sets are defined in
such a way that their union intersects~$\SS$ along a set~$\cI^{\tn{C}}$
which corresponds to Diophantine numbers~$\al$ of
exponent~$2+\tau$, \ie
\begin{equation}	\label{ineqDioph}
\sup \left\{ m^{-2-\tau}\left|\al-\tfrac{n}{m}\right|\ii,\;{\tfrac{n}{m}\in\Q} \right\} < \infty.
\end{equation}
Complementing Herman's regularity result with the above-mentioned holomorphy
result in $\ao q=\ee^{2\pi\I\al} \in\PP \mid \; |\IM\al| \ge
R/200 \af$, we shall be in a position to apply the results of Part~A.

\begin{rem}
Here, in contrast with Problem~\tn{S}, we do not try to reach the optimal
arithmetical condition, which is known to be the Bruno
condition~(\ref{eqBrunoClassic}) as in the Siegel problem, by Risler's result
based on renormalization---see \cite{Ri}, \cite{Cetraro}.
We content ourselves with~$\cI^\tn{C}$, which has still full measure in~$\SS$.
\end{rem}

\medskip


In the following sections, we shall thus recall the results of \cite{MS1} (with some adaptations),
\cite{CM} and~\cite{He} on the dependence on~$q$ of the solution
of Problems~\tn{L}, \tn{S} and~\tn{C}.
%
%
In each case, a sequence of compact sets $(K_j)_{j\in\N}$ is obtained by removing smaller
and smaller open neighbourhoods of the roots of unity, so that:
\begin{itemize}
\item
The union of the $K_j$'s consists of all the points of~$\PP$ except the roots of
unity and the points $\ee^{2\pi\I \al}\in\SS$ at which (\ref{eqLinOpt}),
resp.~(\ref{eqBrunoClassic}), resp.~(\ref{ineqDioph}), fails; 
in other words, $\bigcup K_j = \D\cup\cI\cup\E$
with $\cI = \cI^{\tn{L}}$, $\cI^{\tn{S}}$ or~$\cI^{\tn{C}}$.
\item
The map $q \mapsto h(q,\,.\,)$ belongs to $\gM\bigl( (K_j),(B_j)\bigr)$
with $B_j = H^\infty(\D_{r_j},\C)$ (the space of bounded holomorphic functions
in~$\D_{r_j}$ with values in~$\C$) for a suitable sequence $(r_j)_{j\in\N}$ in the first two cases,
or $B_j=H^\infty\bigl(\{|\eps|<r_j\},\gO(\ov S_{R/2},\C)\bigr)$ in the last case.
\end{itemize}
We shall then see that the hypotheses of Corollary~A stated in
Section~\ref{secStatRes} are fulfilled, which sheds new light on the
relationship between $h_{|\D}$ and~$h_{|\E}$.


\section{Monogenic regularity and quasianalyticity of the solutions of small divisor problems}
\label{secSDdom}
%

\begin{Def}	\label{defComplexDom}
For any subset~$A$ of the real line which is invariant by integer translations,
we set
\begin{equation}	\label{eqComplexDom}
A^\C =  \bigl\{\, z \in\C \mid\; \exists \al_*\in A
\;\text{ such that }\; |\IM z| \ge |\al_* - \RE z| \,\bigr\}
\end{equation}
and
\begin{equation}	\label{eqDomK}
K = \Exp\big(A^\C\big) \cup \{0,\infty\} \;\subset\; \PP,
\end{equation}
where $\Exp \colon z\in\C \mapsto \ee^{2\pi\I z} \in\C^*\subset\PP$.
The set~$K$ is called the {\em complex multiplier domain} associated with~$A$.
Observe that $K\cap\SS = \Exp(A)$.
\end{Def}

We shall use the following domains:
\begin{enumerate}[{\bf (\ref{secSDdom}.1)}]
\item for any $M>\log3$, 
\begin{equation}	\label{eqRealLin}
A_M^{\tn{L}} = \Bigl\{\, \al \in\R\setminus\Q \mid\; \forall k\in\N, \;
\frac{\log m_{k+1}(\al)}{m_{k}(\al)} \le M \,\Bigr\}
\end{equation}
and the corresponding complex multiplier domain $K_M^{\tn{L}}$;
\item
for any $M>0$, 
\begin{equation}	\label{eqRealNLin}
A_M^{\tn{S}} = \{\, \al \in\R\setminus\Q \mid \gB(\al) \le M \,\}
\end{equation}
and the corresponding complex multiplier domain $K_M^{\tn{S}}$,
where 
\begin{equation}	\label{Brunoseries}
\gB(\al) = \sum_{k\ge0} \frac{\log a_{k+1}(\al)}{m_k(\al)} \in [0,+\infty],
\qquad \al\in\R\setminus\Q,
\end{equation}
$\bigl(a_k(\al)\bigr)_{k\ge0}$ denoting the sequence of partial quotients
of~$\al$ (see Appendix~\ref{appCF});
\item for any $M>2\ze(1+\tau)$ (Riemann's zeta function) with a given
$\tau\in(0,1)$,
\begin{equation}	\label{eqRealNLinC}
A_M^{\tn{C}} = \bigl\{\, \al \in\R\setminus\Q \mid 
\forall \tfrac{n}{m}\in\Q,\;
m^{-2-\tau}\left|\al-\tfrac{n}{m}\right|\ii \le M \,\bigr\}
\end{equation}
and the corresponding complex multiplier domain $K_M^{\tn{C}}$.
\end{enumerate}

\begin{rem}
The function~$\gB$ is closely related to the classical Bruno series
$$\hat\gB(\al) = \sum_{k\ge0} \frac{\log m_{k+1}(\al)}{m_k(\al)}$$ which is involved in
the definition~(\ref{eqBrunoNumb}) of the optimal set~$\cI^{\tn{S}}$;
there exists indeed a constant~$C$ such that
$0\le \hat\gB(\al)-\gB(\al) \le C$ for all $\al\in\R\setminus\Q$ (see \cite{CM}).
As a result, for any sequence $M_j \uparrow \infty$, 
\begin{equation}	\label{eqAMjOpti}
\SS\cap \bigcup_{j\in\N} K_{M_j}^{\tn{S}} =
\{\, \ee^{2\pi\I\al} \mid \al\in\R\setminus\Q, \; \hat\gB(\al) < \infty \,\}
= \cI^{\tn{S}}.
\end{equation}
In fact,
$\bigcup_{j\in\N} K_{M_j}^{\tn{a}} = \D\cup\cI^{\tn{a}}\cup\E$
for \tn{a} $=$ \tn{L}, \tn{S} or~\tn{C}.
\end{rem}

\begin{thmB}
Consider Problem~\tn{L} or~\tn{S} with a given $g\in
H^\infty(\D_R)$, $R>0$.
Then the solution $h:\, q\mapsto h(q,\,.\,)$ belongs to
the spaces $\Chol\bigl(K_M^{\tn{L}},B_M)$ or
$\Chol\bigl(K_M^{\tn{S}},B_M)$, 
with $B_M = H^\infty(\D_{R\,\ee^{-4M}})$;
for any sequence $M_j\uparrow\infty$, the solution thus defines a monogenic
function of
$\gM\bigl(K_{M_j}^{\tn{L}},B_{M_j})$ or
$\gM\bigl(K_{M_j}^{\tn{S}},B_{M_j})$.

As for Problem~\tn{C} with $g\in \gO(\ov S_R,\C)$ of zero mean-value and
$\tau\in(0,1)$ given, for each $M>2\ze(1+\tau)$ there exists $r = r(M)>0$ such
that
\begin{multline}	\label{equniqusol}
q\in K_M^{\tn{C}}, \; |\eps|<r(M)\ens\Rightarrow\ens \\
\exists \;\text{solution}\; (\be,h)
=\bigl(\be(q,\eps),h(q,\eps)\bigr)\in\C\times\gO(\ov S_{R/2},\C).
\end{multline}
Moreover,
the function 
$ q \mapsto \bigl(\be(q,\,.\,),h(q,\,.\,)\bigr) $
belongs to the space $\Chol\bigl(K_M^{\tn{C}},B_M\bigr)$ with 
$B_M = H^\infty\bigl(\{|\eps|<r(M)\},\C\times\gO(\ov S_{R/2},\C)\bigr)$
and thus defines a monogenic function of 
$\gM\bigl( (K_{M_j}^{\tn{C}}), (B_{M_j}) \bigr)$
for any sequence $M_j\uparrow\infty$.

For \tn{a} $=$ \tn{L}, \tn{S} or~\tn{C},
the spaces $\gO\bigl(K_M^{\tn{a}},B_M\bigr)$, which contain the spaces
$\Chol\bigl(K_M^{\tn{a}},B_M\bigr)$, are $\gH^1$-quasianalytic relatively
to~$K_M^{\tn{a}}$ and quasianalytic at the interior points
of~$K_M^{\tn{a}}$,
and the spaces $\gM\bigl( (K_{M_j}^{\tn{a}}), (B_{M_j}) \bigr)$ are  
$\gH^1$-quasianalytic relatively to $\D\cup \cI^{\tn{a}} \cup\E$ and
quasianalytic at the points of $\D\cup\E$.
\end{thmB}

\noindent
The rest of the article is devoted to the proof of Theorem~B.

\medskip

The domains $K_M^{\tn{a}}\subset \PP$ are obtained from the arithmetical
conditions~(\ref{eqRealLin}), (\ref{eqRealNLin}) or~(\ref{eqRealNLinC}). 
The connection with Part~A of the article is provided by the following lemma,
which will make it possible to apply Theorem~A and its corollary to the spaces
$\gO\bigl(K_M^{\tn{a}},B\bigr)$ or
$\gM\bigl( (K_{M_j}^{\tn{a}}), (B_{M_j}) \bigr)$.

\begin{lemma}	\label{lemNest}
Let $A\subset\R$ be invariant by integer translations and $K\subset\PP$ the
corresponding complex multiplier domain as in Definition~\ref{defComplexDom}.
If $A$ is closed and has positive Lebesgue measure, then there exists a nested
pair $\bigl(\Ga\IN,\Ga\EX\bigr)$ such that 
\begin{equation}	\label{eqKI}
K = K\bigl(\Ga\IN,\Ga\EX\bigr), \quad \Exp(A) = \cI\bigl(\Ga\IN,\Ga\EX\bigr).
\end{equation}
\end{lemma}

\noindent {\em Proof of Lemma~\ref{lemNest}.}
Let $\Phi:\, x\in\R \mapsto \dist(x,A)$.
This is a $1$-periodic function, which satisfies (with the notation of
Definition~\ref{defComplexDom})
$$
A^\C = \ao x+\I y \mid x,y\in\R,\; |y|\ge\Phi(x) \af.
$$
(Indeed, if $x+\I y\in A^\C$, then according to~(\ref{eqComplexDom}) there is
$\al^*\in A$ such that $|y| \ge |\al_*-x| \ge \Phi(x)$; conversely, if $|y|\ge
\Phi(x)$, just pick $\al_*\in A$ such that $\Phi(x) = |\al_*-x|$.)

Since $\Phi$ is $1$-Lipschitz, we have two rectifiable Jordan curves
$$
\Ga\IN = \bigl\{ \ee^{-2\pi\Phi(x)}\cdot\ee^{2\pi\I x}, \; x\in\R \bigr\},
\quad
\Ga\EX = \bigl\{ \ee^{2\pi\Phi(x)}\cdot\ee^{2\pi\I x}, \; x\in\R \bigr\},
$$
which are easily seen to satisfy~(\ref{eqKI}).

\qed

Notice that if the connected components of~$\R\setminus A$ are denoted
$(\al_\ell,\al'_\ell)$, then the connected components of~$\C\setminus A^\C$
consist of the open diamonds~$\De_\ell$ with bases $(\al_\ell,\al'_\ell)$:
\begin{equation}	\label{eqdiamonds}
\De_\ell = \bigl\{\, x+\I y \mid 
x \in (\al_\ell,\al'_\ell), \;
|y| < \min(x -  \al_\ell, \al'_\ell - y)
\,\bigr\}
\end{equation}

\begin{rem}
Problem~\tn{C} can be slightly generalized as follows: introducing the notation
\begin{equation}	\label{eqdefSab}
S_{a,b} = \ao \th\in\C/\Z \mid a < \IM\th < b \af
\qquad \text{for $-\infty \le a < b \le +\infty$},
\end{equation}
we can consider any $G_{\al,g} = \ID + \al + g$ with $g$ holomorphic
on~$S_{a,b}$ and small enough, and look for $(\be,h)$ such that
$G_{\al-\be,g}\circ h = h\circ G_{\al,0}$, with $h-\ID$ of zero mean-value and
holomorphic in an annulus $S_{a',b'}$ the closure of which is contained
in~$S_{a,b}$.

Problem~\tn{S} now appears as a particular case, via the map $\th\mapsto E(\th)=\ee^{2\pi\I\th}$: 
for any $g\in z^2\C\{z\}$, let 
\[
\ell_g(z) = \frac{1}{2\pi\I}\log\big(1+z\ii
g(z)\big) \in z\C\{z\}
\]
and $\ti g = \ell_g\circ E$.
Observe that 
\[
\ti g(\th) = \sum_{k\ge1}\ti g_k\, \ee^{2\pi\I k\th}
\]
is holomorphic
and bounded in~$S_{a,+\infty}$ for $a$ large enough (the larger~$a$, the smaller
the maximum of~$|g|$),
and 
$$
E \circ (\ID+\al+\ti g) = \big[q(\ID+g)\big]\circ E,
$$
hence linearizing $q(\ID+g)$ by $h\in z+z^2\C\{z\}$ amounts to conjugating
$\ID+\al+\ti g$ to $\ID+\al$ by 
$\ti h = \ID + \frac{1}{2\pi\I} \big(\log\frac{h}{\ID}\big) \circ E$
holomorphic in~$S_{a',+\infty}$ for a certain~$a'$, 
and $\be=0$ in this case.

\end{rem}


\section{Proof of Theorem~B}
\label{secpfthmB}
%

\subsection{Case of the cohomological equation}
%

Let $M>\log3$.
It is shown in \cite[\S2.3]{MS1} that $[0,1]\cap A_M^\tn{L}$ has Lebesgue
measure $\ge 1-\frac{2}{\ee^M-1} > 0$, is totally disconnected, closed and
perfect;
it follows from~(\ref{ineqCVk}) and~(\ref{ineqLawBA}) that this set consists of
points which are ``far enough from the rationals'', namely
$$
\bigcap_{n/m} \{\, \al\in\R \mid |\al-\tfrac{n}{m}| \ge \tfrac{1}{m\,\ee^{M m}} \,\}
\;\subset\; A_M^\tn{L} \;\subset\;
\bigcap_{n/m} \{\, \al\in\R \,|\; |\al-\tfrac{n}{m}|> \tfrac{1}{2m\,\ee^{M m}} \,\}.
$$

It is proved in \cite[\S2.4]{MS1} that the function $q\mapsto h(q,\,.\,)$
which describes the solution of the cohomological
equation~(\ref{eq:cohomological}) belongs to
$\Chol\bigl( K_M^\tn{L}, H^\infty(\D_r) \bigr)$ as soon as $0 < r < R\,\ee^{-3M}$.
The conclusion then follows easily from Lemma~\ref{lemNest} applied to~$A_M^\tn{L}$.

\subsection{Case of the Siegel problem}
%
%

Although the technical proofs of \cite{CM}
depart a lot from those in \cite{MS1},
the results for the set~$A_M^\tn{S}$ and the regularity of the solution of
Problem~\tn{S} shown there are similar to those for~$A_M^\tn{L}$ and Problem~\tn{L}.
In particular, \cite{CM} shows that $A_M^\tn{S}$ is a closed subset of~$\R$, which is also totally
disconnected and perfect. Moreover, we have 
\begin{lemma}	\label{lemAMR}
For every $M>0$, the set $A_{M}^\tn{S}$ has positive Lebesgue measure in~$\R$.
\end{lemma}

\noindent 
The proof is postponed to Appendix~\ref{appPfLem}.

\medskip

It is proved in \cite{CM} that the function $q\mapsto h(q,\,.\,)$ which
describes the solution of the conjugacy equation~(\ref{eqConj}) belongs to
$\Chol\bigl( K_M^\tn{S}, H^\infty(\D_r) \bigr)$ as soon as $0 < r \le
R\,\ee^{-(3+\de)M}$ for any $\de>0$.
The conclusion follows easily from Lemma~\ref{lemNest} applied to~$A_M^\tn{S}$.

\subsection{Case of the local conjugacy problem of complex
maps of the annulus}
\label{secSDpbNLC}

Let $M > 2\ze(1+\tau)$. Observe that
$$
(0,1)\setminus A_M^\tn{C} \;\subset\; \bigcup \big(
\tfrac{n}{m}-\tfrac{M\ii}{m^{2+\tau}}, \tfrac{n}{m}+\tfrac{M\ii}{m^{2+\tau}}
\big)\cap(0,1),
$$
where the union extends to all $(n,m)\in\Z\times\N^*$ with $(n,m)=1$; in fact,
since $M>1$, we can restrict it to $0\le n\le m$, thus the Lebesgue measure of this
set is 
$$
\big|(0,1)\setminus A_M^\tn{C}\big| \le 
2M\ii + \sum_{m\ge2} \sum_{n=1}^{m-1} 2M\ii m^{-2-\tau} <
\frac{2\ze(1+\tau)}{M}.
$$
As a consequence $A_M^\tn{C}$ is closed subset of positive Lebesgue measure of~$\R$ (and
$\cI^\tn{C}$ has full measure).

As alluded to at the end of Section~\ref{secSDdom}, \cite{He} studies the regularity of the
solution of Problem~\tn{C} in
$$
\ti K_M^\tn{C} = K_M^\tn{C} \cap \ao q=\ee^{2\pi\I\al}, |\IM\al|\le R/100 \af.
$$
It is proved there that there exists $r = r(M)>0$ for which (\ref{equniqusol})
holds with~$K_M^\tn{C}$ replaced by~$\ti K_M^\tn{C}$, and that the function 
$ q \mapsto \bigl(\be(q,\,.\,),h(q,\,.\,)\bigr) $
belongs to the space $\Chol\bigl(\ti K_M^{\tn{C}},B_M\bigr)$ with 
\[ B_M = H^\infty\bigl(\{|\eps| < r(M)\},\C\times\gO(\ov S_{R/2},\C)\bigr). \]
To conclude, it is sufficient to extend this regularity property from~$\ti
K_M^\tn{C}$ to~$K_M^\tn{C}$ and to apply Lemma~\ref{lemNest} to~$A_M^\tn{C}$.
If we use the rephrasing of Problem~\tn{C} as equations~(\ref{eqdefbehv})--(\ref{eqfixedv}),
the conclusion thus follows from 
\begin{lemma}
Let $\La>0$, $\Om = \ao q=\ee^{2\pi\I\al}, |\IM\al|> \La \af$, 
$B = \gO\big(\ov S_{R/2},\C\big)$ and
$\cE = 2 +
(\ee^{2\pi R}-1)\ii + \demi\sinh^{-2}(\pi\La)$.
Then:
\begin{enumerate}[(i)]
\item For each $q\in \Om$, $E_q$ is a bounded linear operator of~$B$: $\norm{E_q v}\le\cE\norm{v}$, 
and, given $r'>0$ and a function $(q,\eps)\mapsto v_{q,\eps}\in B$ holomorphic
for $q\in\Om$ and $|\eps|<r'$, the function $(q,\eps)\mapsto E_q v_{q,\eps}\in
B$ is holomorphic too.
%
%
\item There exists $r'>0$ such that, for all $q\in \Om$ and $\eps$ such that 
$|\eps|<r'$,
there exists a unique solution~$v$ of equation~(\ref{eqfixedv}) in~$B$ close to~$0$.
Moreover, this solution~$v$ is a bounded holomorphic function of~$(q,\eps)$.
\end{enumerate}
\end{lemma}

\begin{proof}
As a preliminary, we introduce the following notation:
for $k\in\Z$, $e_k$ will denote the function $\th\mapsto\ee^{2\pi\I k\th}$, and
if
$v = \sum_{k\in\Z} v_k e_k\in B$, 
$$
\Pi^+ v = \sum_{k\ge1} v_k e_k, \qquad
\Pi^- v = \sum_{k\ge1} v_{-k} e_{-k}, \qquad
\Pi_k v = v_k e_k, \qquad k\in\Z.
$$
Since the Fourier coefficients of~$v$ can be computed as 
\[
v_k = \int_{\I\rho}^{1+\I\rho} v(\th)\,\ee^{-2\pi\I k\th}\,\dd\th
\]
for any $\rho\in\big[-\frac{R}{2},\frac{R}{2}\big]$, it follows that 
\begin{equation}	\label{ineqFouriercoeff}
|v_k| \le \ee^{-\pi |k| R} \norm{v}, \qquad k\in\Z,
\end{equation}
hence, with the notation~(\ref{eqdefSab}), $\Pi^\pm v$ is holomorphic in $S^+ =
S_{-{R}/{2},+\infty}$ or $S^- = S_{-\infty,{R}/{2}}$ and is in fact a function
of $z^\pm = \ee^{\pm 2\pi\I\th}$ holomorphic for $|z^\pm|<\ee^{\pi R}$.
More than this, one has
\begin{equation}	\label{ineqPipm}
\begin{aligned}
%
%
&\text{$\Pi^\pm v$ extends continuously to $\ov S^\pm$,}\\
&|\Pi^\pm v(\th)| \le \Big( 2 + \frac{1}{\ee^{2\pi R}-1} \Big)\norm{v}
\ens\text{for $\th\in\ov S^\pm$.}
\end{aligned}
\end{equation}
Indeed, in the case of the `$+$' sign for instance, $\Pi^+ v = v-\Pi_0 v - \Pi^-
v$ where $v-\Pi_0 v$ is holomorphic in~$S^+$, continuous in~$\ov S^+$ and
bounded by $2\norm{v}$, while $\Pi^- v$ is holomorphic in any
neighbourhood~$\cN$ of $\{\IM\th=-R/2\}$ contained in~$S^-$, hence this
representation of~$\Pi^+ v$ gives a continuous extension to $\cN\cap\ov S^+$;
by the maximal modulus principle, the maximum of~$|\Pi^+ v|$ is attained for
$|z^+|=\ee^{\pi R}$, it is thus equal to
$
\max \big\{|\Pi^+ v(\th)|,\:\IM\th=-R/2\big\} \le 2\norm{v} + 
%
%
\max \big\{ |\Pi^- v(\th)|,\: \IM\th=-R/2 \big\}
%
%
$
and (\ref{ineqPipm}) follows from 
$\IM\th=-R/2 \Rightarrow |\ee^{-2\pi\I k\th}|=\ee^{-\pi k R}$
and from~(\ref{ineqFouriercoeff}) applied to $|v_{-k}|$
for $k\ge1$.

We can rephrase~(\ref{ineqFouriercoeff}) and~(\ref{ineqPipm}) as statements
about bounded linear operators:
\begin{equation}	\label{ineqLB}
\Pi_k,\Pi^\pm \in \cL(B), \quad
\norm{\Pi_k}\le1, \quad
\norm{\Pi^\pm}\le 2 + \frac{1}{\ee^{2\pi R}-1}.
\end{equation}

\medskip

{\em (i)} 
Assume $q=\ee^{2\pi\I\al}\in\Om$ with $|q|<1$, \ie\ $\IM \al > \La$ and
$|q|<\ee^{-2\pi\La}$ (we would argue in a symmetric way in the case $|q|>1$),
and let $v\in B$.
Writing $\frac{1}{q^k-1} = -1-\frac{q^k}{1-q^k}$ and
$\frac{1}{q^{-k}-1} = \frac{q^k}{1-q^k}$ for $k\ge1$, we get
\begin{equation}	\label{eqEqPI}
E_q v = -\Pi^+ v + \sum_{k\ge1} \frac{q^k}{1-q^k} (-\Pi_k+\Pi_{-k})v
\end{equation}
On the one hand, $|1-q^k| \ge 1-|q^k| \ge 1-\ee^{-2\pi\La}$ and $|q^k|\le
\ee^{-2\pi k\La}$, 
on the other hand $-\Pi_k+\Pi_{-k}\in\cL(B)$ has operator norm~$\le2$
by~(\ref{ineqLB}), hence (\ref{eqEqPI}) yields a representation of~$E_q$ as an
absolutely convergent series of bounded linear operators and
$$
\norm{E_q} \le \norm{\Pi^+} + 
\sum_{k\ge1} \frac{\ee^{-2\pi k\La}}{1-\ee^{-2\pi\La}}\norm{-\Pi_k+\Pi_{-k}}
\le  2 + \frac{1}{\ee^{2\pi R}-1} +
\frac{2\ee^{-2\pi\La}}{(1-\ee^{-2\pi\La})^2},
$$
which was the desired bound.
Moreover, the operators $\Pi^+$ and $-\Pi_k+\Pi_{-k}$ are independent of
$(q,\eps)$, the coefficients $\frac{q^k}{1-q^k}$ are holomorphic functions
of~$q$ and the above convergence was uniform in~$q$, hence we obtain the
holomorphic dependence on $(q,\eps)$.

\medskip

{\em (ii)} 
Let $C=\max_{\ov S_R}|g|$ and $r' = \frac{R}{8\cE C}$. We shall prove the
statement with this value of~$r'$ by means of the contraction principle.

For any $q\in\Om$ and $\eps\in\C$ such that $|\eps| < r'$, we define a
map~$\FF$ on the ball $\gV = \ao v\in B \mid \norm{v}\le r'C \af$ by the formula
$$
\FF(v)\,: \; \th\in\ov S_{R/2} \ens\mapsto\ens \eps g\big(\th+E_q v(\th)\big)
\qquad\text{for $v\in \gV$}.
$$
The function $\FF(v)$ is well-defined when $v\in\gV$ because, 
for every $\th\in\ov S_{R/2}$, $|E_q v(\th)|\le \cE\norm{v}\le R/8$, hence
\begin{equation}	\label{eqthEqv}
\th+E_q v(\th)\in S_{3R/4}.
\end{equation}
This function clearly belongs to~$B$, with $\norm{\FF(v)}\le|\eps|C\le r'C$.
Therefore, $\FF$ is a self-map of~$\gV$.

Suppose $v_1,v_2\in\gV$. For each $\th\in\ov S_{R/2}$, (\ref{eqthEqv}) and the fact that
$\max_{\ov S_{3R/4}}|g'| \le 4C/R$ imply that 
\[
|\FF(v_2)(\th) - \FF(v_1)(\th)| \le \frac{4C\cE|\eps|}{R} \norm{v_2-v_1}\le
\demi\norm{v_2-v_1}.
\]
We thus get a unique fixed point for~$\FF$, \ie\ a unique solution of~(\ref{eqfixedv}), in~$\gV$.
This solution can be written as the sum of the series
$$
%
\FF(0) + \sum_{k\ge0} \big( \FF^{k+1}(0) - \FF^k(0) \big)
%
$$
which is absolutely convergent in~$B$. 

Let us check that this fixed point depends holomorphically on~$(q,\eps)$ for
$q\in\Om$ and $|\eps| < r'$.
Since $\norm{\FF^{k+1}(0) - \FF^k(0)}\le \frac{1}{2^k}\norm{\FF(0)}$,
it is sufficient to check that the functions $(q,\eps)\mapsto \FF^k(0)$ are
holomorphic. 
This follows by induction from the fact that, if $(q,\eps)\mapsto
v_{q,\eps}\in\gV$ is holomorphic for $q\in\Om$ and $|\eps| < r'$, then
$(q,\eps)\mapsto \FF(v_{q,\eps})$ is holomorphic too
(proof: since $B$ is a Banach algebra, $\FF(v_{q,\eps})$ can be written as the
sum of the uniformly convergent series
$\eps\sum_{\ell\ge0} g_\ell (E_q v_{q,\eps})^\ell $, 
where $g_\ell = \frac{1}{\ell!}{g^{(\ell)}}_{|\ov S_{R/2}}$, 
$\norm{g_\ell}\le C (2/R)^\ell$,
$\norm{(E_q v_{q,\eps})^\ell}\le (\cE r' C)^\ell$).

\end{proof}


\appendix

\section{Appendix}



\subsection{Continued fractions}	\label{appCF}
%
%
We indicate here our notations and a few facts that we use in
Part~B, referring the reader \eg\ to \cite{HW} for an exposition of
the theory.

Given $x \in \R \setminus \Q$, we set
$a_0 = [x]\in\Z$, $\xi_0 = x - [x]\in (0,1)$, and inductively
$a_{k+1} = [ \xi_k\ii] \in\N^*$, $\xi_{k+1} = \xi_k\ii - [\xi_k\ii]\in (0,1)$, 
hence
$$
x = a_0 + \cfrac{1}{ a_1 +
          \cfrac{1}{\ens \;\ddots\; +
%
             \cfrac{1}{ a_k + \xi_k}}}.
%
$$
Dropping $\xi_k$ in the last expression, we get a rational number called the $k$th
{\em convergent} of~$x$ and denoted by $[a_0,a_1,\ldots,a_k]$.
We denote by $\frac{n_k}{m_k}$ the reduced expression of this rational.

The $a_k$'s are called {\em partial quotients} of~$x$; notice that, for $k\ge1$, they are all
positive integers.
We sometimes write $a_k(x)$, $n_k(x)$, $m_k(x)$, considering the partial
quotients and the convergents as functions of $x\in\R\setminus\Q$.

Numerators and denominators of the convergents can be obtained from the recursive formulas
\begin{equation}	\label{eqInducnkmk}
n_k = a_k n_{k-1} + n_{k-2}, \quad m_k = a_k m_{k-1} + m_{k-2}, 
\end{equation}
with initial conditions $n_{-1}=1$, $n_{-2}=0$, $m_{-1}=0$, $m_{-2}=1$.
They satisfy
\begin{gather}	
\label{eqBezout}
m_k n_{k-1} - n_k m_{k-1} = (-1)^k, \\
		\label{eqxxik}
x = \frac{n_k + \xi_k n_{k-1}}{m_k + \xi_k m_{k-1}}.
\end{gather}
The formulas~(\ref{eqInducnkmk}) imply that $(m_k)_{k\ge1}$ is an increasing
sequence of positive integers bounded from below by the Fibonacci numbers:
\begin{equation}	\label{eqFibo}
m_{k}\ge F_{k+1} \ge \Phi^{k-1}, \quad 
F_k = \frac{\Phi^k + (-1)^{k+1}\ph^k}{\sqrt{5}},
\qquad k\ge0,
\end{equation}
with $\ph = \frac{\sqrt{5}-1}{2} \in (0,1)$ and $\Phi = 1/\ph = 1+\ph$ (golden ratio).

The convergents $n_k/m_k$ converge to~$x$ at least geometrically; more precisely,
\begin{equation}	\label{ineqCVk}
\frac{1}{2 m_{k+1}} < |m_k x - n_k| < \frac{1}{m_{k+1}}, \qquad  k\ge1,
\end{equation}
$\frac{n_k}{m_k} < x < \frac{n_{k+1}}{m_{k+1}}$ for even $k$ (reverse
inequalities for odd~$k$),
$\left| \frac{n_{k+1}}{m_{k+1}} - \frac{n_k}{m_k} \right| = \frac{1}{m_k m_{k+1}}$.
According to the classical {\em law of best approximation}, if a rational
$\frac{n}{m}$ does not belong to the sequence of convergents and if $k\in\N^*$,
then
\begin{equation}	\label{ineqLawBA}
m \le m_{k+1}  \quad\Rightarrow \quad
| m x - n | > | m_k x - n_k |.
\end{equation}
As a partial converse to~(\ref{ineqCVk}), we mention the fact that if a
rational $n/m$ satisfies $| m x - n | < \frac{1}{2m}$, then it necessarily
belongs to the sequence of convergents of~$x$.


\subsection{Proof of Lemma~\ref{lemAMR}}
\label{appPfLem}
%
%
For $\tau\ge0$, we introduce the Diophantine numbers of~$(0,1)$ of exponent
$2+\tau$:
$$
\DCgt = \bigl\{\, x\in (0,1) \mid 
\forall \tfrac{n}{m}\in\Q\cap(0,1),\;
\left| x - \frac{n}{m} \right| \ge \frac{\ga}{m^{2+\tau}} \,\bigr\},
\qquad \ga>0
$$
(the smaller~$\ga$, the larger $\DCgt$).
The set of ``constant-type points'' of~$(0,1)$ can be defined as
$\bigcup_{\ga>0} \DC_{\ga,2}$, but it won't be of any use for us here because
it has zero measure, whereas it is well-known that the larger sets $\bigcup_{\ga>0}
\DC_{\ga,\tau}$ have full measure for all $\tau>0$ (see Section~\ref{secSDpbNLC}; $\DCgt$ is
essentially $A_{\ga\ii}^\tn{S}\cap(0,1)$).
It follows from~(\ref{ineqCVk}) that all these numbers satisfy the Bruno
condition~(\ref{eqBrunoClassic}), which we saw is equivalent to $\gB(x)<\infty$.
For every $M>0$, we shall find a positive measure subset of them for which
$\gB(x)\le M$.
The starting point is the observation (obvious from the
definition~(\ref{Brunoseries})) that 
$$
a_1(x)=\cdots=a_k(x)=1 \ens\Rightarrow\ens
\gB(x) = \sum_{\ell\ge k} \frac{\log a_{\ell+1}(x)}{m_\ell(x)}.
$$

Given $k\in\N^*$ and $a_1,\ldots,a_k\in\N^*$, we set
$\frac{n_\ell}{m_\ell}=[0,a_1,\ldots,a_\ell]$ for $0\le\ell\le k$;
the corresponding ``interval of rank~$k$'' (see \cite{Khin}) is then defined as
$$
I(a_1,\ldots,a_k) = 
\Bigl\{ \frac{n_k + \ze n_{k-1}}{m_k + \ze m_{k-1}}, \; \ze\in (0,1) \Bigr\}
= \left| \begin{aligned}
(\tfrac{n_k}{m_k}, \tfrac{n_k+n_{k-1}}{m_k+m_{k-1}}) \quad\text{if $k$ is even} \\[1ex]
(\tfrac{n_k+n_{k-1}}{m_k+m_{k-1}}, \tfrac{n_k}{m_k}) \quad\text{if $k$ is odd.}
\end{aligned} \right.
$$
Then $I(a_1,\ldots,a_k) \setminus \Q = \{\, x\in(0,1)\setminus\Q \mid 
a_\ell(x) = a_\ell \; \text{for $1\le \ell \le k$} \,\}$
(compare with~(\ref{eqxxik})).

The case where $a_\ell=1$ for $1\le\ell\le k$ is related to the golden ratio
conjugate $\ph=[0,1,1,\ldots]$, for which
$$
n_\ell(\ph) = F_\ell, \quad m_\ell(\ph) = F_{\ell+1}, \qquad \ell\in\N.
$$
We then abbreviate the previous notation:
$$
I_k = I(\,\underset{\text{$k$ times}}{\underbrace{1, 1, \ldots, 1}}\,),
$$
which is an open interval of length $|I_k| = \frac{1}{F_{k+1}F_{k+2}}$ 
(from now on, we simply denote by~$|\,.\,|$ the Lebesgue measure on~$\R$).

With these preliminaries, Lemma~\ref{lemAMR} follows from the more precise 

\begin{lemma}

\smallskip

\noindent (i)\; 
For each $M>0$, $\tau\ge0$ and $\ga>0$, there exists $\bar k\in\N^*$ such that
$$
k\ge\bar k  \ens\Rightarrow\ens  \DCgt \cap I_k \subset A_M^\tn{S}.
$$

\smallskip

\noindent (ii)\; 
There exists $\ga^*>0$ such that, for $\tau\ge1$ and $0<\ga<\ga^*$,
$$
k\ge2  \ens\Rightarrow\ens  
\left|\DCgt \cap I_k\right| > \Bigl( 1 - \frac{\ga}{\ga^*} \Bigr) \left|I_k\right|.
$$
\end{lemma}

\noindent {\it Proof of (i).}
Let $M>0$, $\tau\ge0$, $\ga>0$ and $k\in\N^*$. 

If $x\in\DCgt$ and $\ell\in\N^*$, then
$\frac{\ga}{m_\ell^{1+\tau}} \le | m_\ell x - n_\ell | < \frac{1}{m_{\ell+1}}$
by~(\ref{ineqCVk}) (denoting by $\frac{n_1}{m_1},\frac{n_2}{m_2},\ldots$
the convergents of~$x$), and $m_{\ell+1} = a_{\ell+1} m_{\ell} + m_{\ell-1} >
a_{\ell+1} m_{\ell}$, hence
$$ a_{\ell+1} < \frac{m_\ell ^\tau}{\ga}. $$
If we assume moreover $x\in I_k$, then $a_1=\cdots=a_k=1$ and 
$$
\gB(x) < \log(\ga\ii) \sum_{\ell\ge k} \frac{1}{m_\ell} +
 \tau \sum_{\ell\ge k} \frac{\log m_\ell}{m_\ell} 
< \log(\ga\ii) \sum_{\ell\ge k} \frac{1}{m_\ell} +
 \tau \sum_{\ell\ge k} \frac{2}{\sqrt{m_\ell}}. 
$$
The inequalities~(\ref{eqFibo}) thus yield
$\gB(x) < \dst\sum_{\ell\ge k} \left(\log(\ga\ii) \ph^{\ell-1} + 2\tau \ph^{\frac{\ell-1}{2}}\right)$,
which can be made $\le M$ by choosing~$k$ large enough since the series is convergent.

\bigskip

\noindent {\it Proof of (ii).}
Let $\tau \ge 1$, $k\ge2$ and suppose
$0<\ga<\frac{1}{3}$ to begin with.
We have
$$
I_k \setminus \DCgt = \bigcup_{{n}/{m}\in\Q\cap(0,1)} J_{n/m} \cap I_k,
\quad\text{with}\ens J_{{n}/{m}} = \left( \frac{n}{m}-\frac{\ga}{m^{2+\tau}},
\frac{n}{m}+\frac{\ga}{m^{2+\tau}} \right).
$$
Our goal is to ensure $| I_k \setminus \DCgt | < \frac{\ga}{\ga^*}|I_k|$ for suitable~$\ga^*$.
Obviously,
\begin{multline}	\label{eqInclus}
I_k \setminus \DCgt \subset \bigcup_{{n}/{m}\in\Q_{\ga,\tau,k}} J_{n/m}, \\
\quad
\text{with}\ens \Q_{\ga,\tau,k} = \left\{\, \tfrac{n}{m}\in\Q\cap(0,1) \mid
J_{n/m} \cap I_k \neq \emptyset \,\right\}.
\end{multline}
From now on, we shall denote by $\frac{n_1}{m_1},\frac{n_2}{m_2},\ldots$
the convergents of~$\ph$, specifying the argument only when referring to a point
possibly different from~$\ph$; thus, $\frac{n_\ell}{m_\ell} \equiv \frac{F_\ell}{F_{\ell+1}}$.
We first establish that any ${n}/{m}\in\Q_{\ga,\tau,k}$ has $m\ge m_k$.

Indeed, for such a rational (which we suppose written in least terms), we may
choose an irrational $x$ in the non-empty open interval $I_k \cap J_{n/m}$, for
which
$$
\left| x - \frac{n}{m} \right| < \frac{\ga}{m^{2+\tau}} < \frac{1}{2 m^2}.
$$
The ``partial converse to~(\ref{ineqCVk})'' of the end of Appendix~\ref{appCF}
yields $\ell\in\N$ such that $\frac{n}{m}=\frac{n_\ell(x)}{m_\ell(x)}$.
Let us check that $\ell$ cannot be $\le k-2$ by contradiction: 
this would lead to 
\[
\frac{n_\ell(x)}{m_\ell(x)} = \frac{n_\ell}{m_\ell}
\quad\text{and}\quad
\frac{n_{\ell+2}(x)}{m_{\ell+2}(x)} = \frac{n_{\ell+2}}{m_{\ell+2}}
\]
(since $x\in I_k$ and $\ell+2\le k$),
with both of these rationals on the same side of~$x$, whence
\[
\left| x - \tfrac{n}{m} \right| > \left| 
\frac{n_{\ell+2}}{m_{\ell+2}} - \frac{n_\ell}{m_\ell}
\right| = \frac{1}{m_{\ell}m_{\ell+2}}
\]
(the last identity results from~(\ref{eqBezout}) applied twice),
and since 
\[
\frac{\ga}{m^{2+\tau}} = \frac{\ga}{m_{\ell}^{2+\tau}} 
> \left| x - \tfrac{n}{m} \right|,
\] 
we would get
$\ga > \frac{m_{\ell}^{1+\tau}}{m_{\ell+2}}$, which is easily seen to be 
$\ge 1/3$.

In fact, $\ell$ cannot be equal to~$k-1$ either, for this would lead to
$\tfrac{n_k+n_{k-1}}{m_k+m_{k-1}}$ lying between
$\frac{n}{m}=\frac{n_{k-1}}{m_{k-1}} \notin I_k$ 
and the points of~$I_k$ (because this point is the $(k+1)$th convergent of~$\ph$
and thus lies on the same side of~$\ph$ as $\frac{n_{k-1}}{m_{k-1}}$),
hence
\[
\left| \frac{n_{k-1}}{m_{k-1}} - \tfrac{n_k+n_{k-1}}{m_k+m_{k-1}} \right|
< \left| x - \frac{n_{k-1}}{m_{k-1}} \right| < \frac{\ga}{m_{k-1}^{2+\tau}},
\]
where the left-hand side is 
\[
\frac{1}{m_{k-1}(m_{k-1}+m_k)} > \frac{1}{2m_{k-1}m_k},
\]
whence $\ga > \frac{m_{k-1}^{1+\tau}}{2 m_k} > 1$, a contradiction.
Hence, $\ell \ge k$ and $m = m_\ell(x) \ge m_{k}(x) = m_k$ (because $x\in
I_k$).

\medskip

Now, let $p_m$ denote, for any $m\in\N^*$, the number of integers~$n$ such that
$n/m\in\Q_{\ga,\tau,k}$. We just saw that $p_m=0$ for $m<m_k$; we now prove
that $p_m < 3 + m \left| I_k \right|$.

Suppose indeed $p_m\ge 3$ for a given $m$ (which is necessarily $\ge m_k$), and
let $n^-$ and $n^+$ denote the minimal and maximal numerators such that
$n/m\in\Q_{\ga,\tau,k}$; thus $p_m = n^+-n^-+1$.
As the intervals $J_{n/m}$ are disjoint (they are separated by a distance $\frac{1}{m}
- \frac{2\ga}{m^{2+\tau}} > \frac{1}{m} - \frac{1}{m^2}>0$), we see that 
both $\frac{n^-+1}{m}$ and $\frac{n^+-1}{m}$ belong to~$I_k$, hence 
$\frac{n^+-n^--2}{m} < \left| I_k \right|$, which yields the desired inequality.

\medskip

Therefore, (\ref{eqInclus}) implies that
$$
\left| I_k \setminus \DCgt \right| \le \sum_{m\ge m_k} \frac{2\ga p_m}{m^{2+\tau}} 
< \bigl( 
6 Z_{1+\tau}(m_k) + 2 \left| I_k \right| Z_{\tau}(m_k)
\bigr) \ga,
$$
with the notation 
\[
Z_\al(N) = \sum_{m\ge N} \frac{1}{m^{1+\al}} <
\frac{1}{\al(N-1)^\al} 
\quad \text{for $N\ge2$.}
\]
Observe that $|I_k| = \frac{1}{m_k(m_k+m_{k-1})}> \frac{1}{2 m_k^2}$, hence
$\frac{Z_{1+\tau}(m_k)}{|I_k|} < 2 m_k^2 Z_2(m_k) < 4$,
while $Z_\tau(m_k) \le Z_1(2) \le 1$,
thus $\ga^*=1/26$ will do.
\qed

\vspace{8pt}

\noindent {\em Remark}.\;
A simple adaptation of the above proof of~{\it(ii)} yields the following more
general result:
for each $\bar\ga>0$ and $\bar\tau\ge0$, 
for each $\bar x = [0,a_1,a_2,\ldots]\in\DC_{\bar\ga,\bar\tau}$
and $\tau \ge \max(1,\bar\tau)$,
\begin{multline*}
0 <\ga < \min\Bigl(\frac{1}{26},\frac{\bar\ga}{2}\Bigr)
\ens\text{and}\ens k\ge2 \ens\Rightarrow\ens \\
\left|\DCgt \cap I(a_1,\ldots,a_k)\right| > 
\Bigl( 1 - \frac{\ga}{\ga^*} \Bigr) \left|I(a_1,\ldots,a_k)\right|.
\end{multline*}




\bigskip

%
\noindent {\em Acknowledgements.}
We warmly thank A.~Avila who encouraged us to study more deeply the
quasianalyticity issue after having heard us reporting on \cite{MS1} during the
trimester on dynamical systems held at the Centro Ennio de Giorgi (Pisa) in
2002, G.~Forni who on the same occasion attracted our attention to
Privalov's theorem, which was the starting point of this work,
and Carlo Carminati for fruitful conversations.


\bigskip

\noindent
{\bf Stefano Marmi}

\noindent
Scuola Normale Superiore; Piazza dei Cavalieri 7, 56126 Pisa, Italy\\
(e-mail: {\tt s.marmi@sns.it})

\medskip

\noindent
{\bf David Sauzin}

\noindent
Scuola Normale Superiore di Pisa and\\
Institut de m\'ecanique c\'eleste, CNRS; 77 av.\ Denfert-Rochereau, 75014 Paris, France\\
(e-mail: {\tt sauzin@imcce.fr})

\end{document}